\documentclass[11pt,a4paper]{amsart}

\usepackage{amsmath,amsthm,amssymb,latexsym,a4wide,tikz,tikz-cd,multicol}
\usepackage{arydshln,multirow}
\usepackage{tikz-qtree,tikz-qtree-compat}
\usepackage{mathtools}
\tikzset{font=\small}

\newtheorem{theorem}{Theorem}[section]
\newtheorem{lemma}[theorem]{Lemma}
\newtheorem{corollary}[theorem]{Corollary}
\newtheorem{proposition}[theorem]{Proposition}

\theoremstyle{definition}
\newtheorem{remark}[theorem]{Remark}
\newtheorem{definition}[theorem]{Definition}
\newtheorem{example}[theorem]{Example}

\usepackage{tikz}
\usepackage{soul}
\usepackage{enumerate}
\usetikzlibrary{matrix,arrows}
\usetikzlibrary{positioning}
\usetikzlibrary{patterns}

\usepackage[active]{srcltx}

\usepackage{enumerate}
\numberwithin{equation}{section}

\title[Quotients of the Booleanization of an inverse semigroup]{Quotients of the Booleanization of an inverse semigroup}

\author{Ganna Kudryavtseva}
\address{G. Kudryavtseva: University of Ljubljana,
Faculty of Civil and Geodetic Engineering, Jamova cesta~2, SI-1000 Ljubljana, Slovenia / Institute of Mathematics, Physics and Mechanics, Jadranska ulica 19, SI-1000 Ljubljana, Slovenia}
\email{ganna.kudryavtseva\symbol{64}fgg.uni-lj.si}


\thanks{The author was partially supported by  ARRS grant P1-0288.}

\begin{document}

\begin{abstract}  
We introduce $X$-to-join representations of inverse semigroups which are a relaxation of the notion of a cover-to-join representation. We construct the universal $X$-to-join  Booleanization of an inverse semigroup $S$ as a weakly meet-preserving quotient of the universal  Booleanization ${\mathrm B}(S)$ and show that all such quotients of ${\mathrm B}(S)$ arise via $X$-to-join representaions.  As an application, we provide groupoid models for the intermediate boundary quotients of the $C^*$-algebra of a Zappa-Sz\'ep product right LCM semigroup by  Brownlowe, Ramagge, Robertson and Whittaker.
\end{abstract}

\maketitle

\section{Introduction}\label{s0:introduction}

A Boolean inverse semigroup is an inverse semigroup whose idempotents form a Boolean algebra and which has joins of compatible pairs of elements (see Section \ref{subs:distr_bool}).  In a Boolean inverse semigroup, the multiplication distributes over finite compatible joins. The Booleanization ${\mathrm B}(S)$ of an inverse semigroup $S$ is the universal Boolean inverse semigroup with respect to all representations of $S$ in Boolean inverse semigroups (see Section \ref{subs:universal}). It is known \cite{Lawson20} that ${\mathrm B}(S)$ can be obtained as the dual Boolean inverse semigroup of  Paterson's universal groupoid ${\mathcal{G}}(S)$ of $S$ under non-commutative Stone duality \cite{Lawson12} between Boolean inverse semigroups and Stone groupoids. 

Tight representations of inverse semigroups \cite{Exel09, Exel:comb, Exel19} are useful in the study of various $C^*$-algebras generated by partial isometries, see, e.g., \cite{BKQS19, DM, Exel:comb, EGS12, EP17, OP19, Starling15} and are closely related to cover-to-join representations \cite{DM, Lawson12, LV19} which are more simpler defined but equally useful: universal objects based on them are isomorphic to those defined on tight representations \cite{Exel19}. The tight Booleanization ${\mathrm B}_{tight}(S)$ of an inverse semigroup $S$ is the universal Boolean inverse semigroup with respect to tight  representations (or, equivalently, with respect to cover-to-join representations) of $S$ in Boolean inverse semigroups. Under the non-commutative Stone duality, it is dual to Exel's tight groupoid ${\mathcal{G}}_{tight}(S)$ of $S$. The latter groupoid is the restriction of ${\mathcal{G}}(S)$ to the tight spectrum ${\widehat{E(S)}_{tight}}$ of the semilattice of idempotents $E(S)$ of $S$ which is a closed and invariant subset of the spectrum $\widehat{E(S)}$. Equivalently, ${\mathrm B}_{tight}(S)$ is a quotient of ${\mathrm B}(S)$ by certain relations, which can be looked at as abstract Cuntz-Krieger relations, see~\cite{LV19}.

For any closed and invariant subset ${\mathcal X}$ of $\widehat{E(S)}$ one can similarly consider the respective restriction of ${\mathcal{G}}(S)$, and the non-commutative Stone duality assigns it a Boolean inverse semigroup being a quotient of ${\mathrm B}(S)$. The present work was inspired by the observation that, besides the tight spectrum of $E(S)$, there are other interesting closed invariant subsets of $\widehat{E(S)}$. Equivalently, besides tight representations, there are other naturally arising representations of~$S$.  For example, prime representations generalize proper morphisms from distributive lattices to Boolean algebras (see Proposition \ref{prop:1}) similar to the way that proper tight representations generalize morphisms between Boolean algebras  (see \cite[Proposition 2.10]{Exel09} and Proposition~\ref{prop:gba_tight}). We also provide examples that come from the theory of $C^*$-algebras, in particular from the theory of intermediate boundary quotient $C^*$-algebras of  Zappa-Sz\'ep product right LCM semigroups by  Brownlowe, Ramagge, Robertson and Whittaker \cite{BRRW}. Using suitable representations (which are $X$-to-join representations defined in this paper, and differ from tight representations), we provide groupoid models for the additive and the multiplicative intermediate boundary quotients $C^*_A(U\bowtie A)$ and $C^*_U(U\bowtie A)$ of the universal $C^*$-algebra $C^*(U \bowtie A)$  of the Zappa-Sz\'ep product $U \bowtie A$, see Section \ref{sect:intermediate}. 

Other examples from the theory of $C^*$-algebras are given in the recent work by Exel and Steinberg \cite{ES19}  where groupoid models of the Matsumoto $C^*$-algebra  and the Carlsen-Matsumoto $C^*$-algebra are constructed. The groupoids are obtained as restrictions (to closed invariant subsets that differ form the tight spectrum) of the universal groupoid of a relevant inverse semigroup. 

The structure of the paper is as follows. In Section \ref{s:preliminaries} we provide preliminaries. Then in Section \ref{sec:X} we define and study $X$-to-join representations of semilattices and inverse semigroups: in Subsection \ref{subs:pitight} we discuss the connection of $X$-to-join representations of semilattices with $\pi$-tight representations from \cite{ES18} and in Theorem \ref{th:booleanization2} we invoke the non-commutative Stone duality to prove that the universal $X$-to-join Booleanization of an inverse semigroup $S$ (where $X$ is invariant) is realized as the dual Boolean inverse semigroup of a suitable restriction of ${\mathcal G}(S)$. In Section \ref{sect:examples} we  discuss prime and core representations of inverse semigroups as examples of $X$-to-join representaitons. Further, in Section \ref{sect:generators} we give a presentation of the universal $X$-to-join Booleanization ${\mathrm B}_X(S)$ of $S$ via generators and relations and in Section \ref{sect:all} we show that all weakly meet-preserving quotients of ${\mathrm B}(S)$ are of the form ${\mathrm B}_X(S)$ for a suitable $X\subseteq E(S)\times {\mathcal P}_{fin}(E(S))$, see Theorem~\ref{th:quotients2} and Corollary \ref{cor:quotients1}. Finally, in Section \ref{sect:intermediate} we approach intermediate boundary quotients of $C^*$-algebras of Zappa-Sz\'ep product right LCM semigroups by Brownlowe, Ramagge, Robertson and Whittaker  \cite{BRRW} via $X$-to-join representations. Let $P$ be a right LCM semigroup and ${\mathcal S}_P$ its left inverse hull. Starling \cite{Starling15} proved that the boundary quotient ${\mathcal Q}(P)$ of the universal $C^*$-algebra $C^*(P)$ of $P$ is a groupoid $C^*$-algebra with the underlying groupoid being the universal tight groupoid of the inverse semigroup ${\mathcal S}_P$. Adapting the arguments of \cite{Starling15} we prove that the additive and the multiplicative intermediate boundary quotients $C^*_A(U\bowtie A)$ and $C^*_U(U\bowtie A)$ (where $P=U\bowtie A$ is a Zappa-Sz\'ep product, see Section \ref{sect:intermediate} for details) of $C^*(U\bowtie A)$ are groupoid $C^*$-algebras where the groupoids are the dual groupoids of suitable $X$-to-join Booleanizations of ${\mathcal S}_P$, see Theorem \ref{th:isomorphisms} and Corollary~\ref{cor:iso}. 

\section{Preliminaries}\label{s:preliminaries}

\subsection{First definitions and conventions} For all undefined notions form inverse semigroup theory we refer the reader to  \cite{Lawson}. For an inverse semigroup $S$, its semilattice of idempotents is denoted by $E(S)$. The {\em natural partial order} on $S$ is given by $a\leq b$ if and only if there is $e\in E(S)$ such that $a=be$. All inverse semigroups that we consider have a bottom element~$0$ which is the minimum element with respect to $\leq$. For $a\in S$ we put ${\mathbf d}(a) = a^{-1}a$ and ${\mathbf r}(a) = aa^{-1}$. Elements $a,b\in S$ are called {\em compatible}, denoted $a\sim b$, if $a{\mathbf d}(b) = b{\mathbf d}(a)$ and ${\mathbf r}(b)a = {\mathbf r}(a)b$, or, equivalently, if the elements $a^{-1}b$ and $ab^{-1}$ are both idempotents.

A {\em semilattice} is  an inverse semigroup $S$ such that $S=E(S)$. The natural partial order on a semilattice $E$ is given by $a\leq b$ if and only if $a=ab$. It follows that $ab$ is the greatest lower bound $a\wedge b$ of  $a$ and $b$.  Hence a semilattice can be characterized as a poset $E$ with zero such that for any $a,b\in E$ the greatest lower bound $a\wedge b$ exists in $E$.  

By a {\em Boolean algebra} we mean what is often called a {\em generalized Boolean algebra}, that is, a relatively complemented distributive lattice with the bottom element $0$. A {\em unital Boolean algebra} is a Boolean algebra with a top element, $1$. 

If $Y$ is a poset and $A\subseteq Y$, the {\em downset} of $A$ is the set $A^{\downarrow} = \{b\in Y\colon b\leq a \text{ for some } a\in A\}$, and the {\em upset} of $A$ is the set $A^{\uparrow} = \{b\in Y\colon b\geq a \text{ for some } a\in A\}$. If $A=\{a\}$ we write $a^{\downarrow}$ and $a^{\uparrow}$ for $A^{\downarrow}$ and $A^{\uparrow}$, respectively.

A {\em Stone space} is a Hausdorff space which has a basis of compact-open sets.  Stone spaces are locally compact.
A map between topological spaces is called {\em proper} if inverse images of compact subsets are compact.

\subsection{Representations and proper representations of semilattices}  

\begin{definition}\label{def:repr}
A {\em representation} of a semilattice $E$ in a Boolean algebra $B$ is a map $\varphi\colon E\to B$ such that:
\begin{enumerate}
\item $\varphi(0)=0$;
\item  for all $x,y\in E$:  $\varphi(x\wedge y) = \varphi(x)\wedge \varphi(y)$.
\end{enumerate}
\end{definition}

\begin{definition} \label{def:proper} A representation $\varphi\colon E\to B$ of a semilattice $E$ in a Boolean algebra $B$ is called {\em proper} if for any $y\in B$ there are  $x_1, \dots, x_n\in E$ such that $\varphi(x_1) \vee \dots \vee \varphi(x_n)\geq y$.
\end{definition}

If $E$ admits the structure of a lattice and the representation $\varphi$ preserves  the operation $\vee$ then $\varphi\colon E\to B$ is proper if for any $y\in B$ there is some $x\in E$ such that $\varphi(x) \geq y$, a notion that appears in~\cite{Doctor}.
Proper representations of semilattices appear in \cite{Exel19} as {\em non-degenerate} representations. If $\varphi\colon E\to B$ is a representation then restricting its codomain to the ideal of $B$ generated by $\varphi(E)$ one obtains a proper representation.
In this paper we will consider exclusively proper representations of semilattices in Boolean algebras.

\begin{definition}
A {\em character} of a semilattice $E$ is a non-zero representation $E\to \{0,1\}$. Equivalently, it is a proper representation  $E\to \{0,1\}$.
\end{definition}

\subsection{The universal Booleanization of a semilattice} Let $E$ be a semilattice. 
A standard argument (to be found, e.g., in \cite[Lemma~34.1]{GH}) shows that the set of all characters $\widehat{E}$ is a closed subset with zero removed of the product space $\{0,1\}^{E}$ and thus, endowed with the subspace topology, is a Stone space. For any $a\in E$ we put 
$$M_a=\{\varphi\in \widehat{E}\colon \varphi(a)=1\}.$$
 The basis of the subspace topology of $\widehat{E}$ are the sets $M_a$, $a\in E$, together with the sets
$$
M_{a;b_1,\dots b_k} = \{\varphi\in \widehat{E}\colon \varphi(a)=1, \varphi(b_1) = \dots = \varphi(b_k) = 0\},
$$
where $k\geq 1$, $a, b_1,\dots, b_k \in E$ and $b_1,\dots, b_k\leq a$. This topology is called the {\em patch topology}.  The set of compact-opens of the space  $\widehat{E}$ forms a Boolean algebra denoted by ${\mathrm B}(E)$ and called the {\em Booleanization} (or  the {\em universal Booleanization}) of $E$. The map $\iota_{{\mathrm B}(E)}\colon E\to {\mathrm B}(E)$ given by $a\mapsto M_a$ is a proper representation of $E$ in ${\mathrm B}(E)$, and, moreover, $\varphi(E)$ generates ${\mathrm B}(E)$ as a Boolean algebra.

\subsection{The tight Booleanization of a semilattice}

Let $E$ be a semilattice. 

\begin{definition} A  subset $Z\subseteq E$ is called a {\em cover} of $x\in E$ if for every $y\in E$ satisfying $y\leq x$ and $y\neq 0$ there is some $z\in Z$ such that $y \wedge z\neq 0$. 
\end{definition} 

The following definition is a variation of the definition due to Donsig and Milan \cite{DM}.

\begin{definition} Let $E$ be a semilattice and $B$ a Boolean algebra. A representation $\varphi\colon E\to B$ is called {\em cover-to-join} provided that it is proper and for any $x\in E$ and for any finite cover $Z$ of $x$ we have:
$$
\varphi(x) = \bigvee_{z\in Z} \varphi(z).
$$
\end{definition}

Cover-to-join representations of semilattices in Boolean algebras are closely related to {\em tight} representations \cite{Exel19} introduced by Exel in \cite{Exel09} for the case where the codomain is a unital Boolean algebra. The definition of a tight representation is a bit more involved than that of a cover-to-join representation and we do not provide it here, because for  our purposes it sufficies to work with cover-to-join representations.  A proper representation of a semilattice $E$ in a  Boolean algebra $B$ is tight if and only if it is cover-to-join, see Exel \cite{Exel19} for details.

\begin{definition} A {\em morphism} between Boolean algebras is a map that preserves the operations $0$, $\vee$, $\wedge$ and $\setminus$ and which is proper.
\end{definition}

We record the following extension of \cite[Proposition 2.10]{Exel09}.

\begin{proposition} \label{prop:gba_tight} Suppose that the semilattice $E$ admits the structure of a Boolean algebra. Then a proper representation $\varphi\colon E\to B$ (where $B$ is a Boolean algebra) is tight if and only if it is a morphism between  Boolean algebras.
\end{proposition}

 \begin{remark}
According to Exel \cite{Exel09} a character $\varphi\colon E\to \{0,1\}$ is {\em tight} if it belongs to the closure of the space of ultra characters of $E$. It follows from \cite[Corollary 4.3]{Exel19} that a a character $\varphi\colon E\to \{0,1\}$ is tight if and only if it is cover-to-join. 
\end{remark}

Let $\widehat{E}_{tight}$  denote the Stone space of tight characters with the subspace topology inherited from $\widehat{E}$. 
The set of compact-opens of the space  $\widehat{E}_{tight}$ forms a Boolean algebra denoted by ${\mathrm B}_{tight}(E)$ and called the {\em tight Booleanization} (or the {\em cover-to-join Booleanization}) of $E$. The map $\iota_{{\mathrm B}_{tight}(E)}\colon E\to {\mathrm B}_{tight}(E)$ given by $a\mapsto M_a\cap \widehat{E}_{tight}$ is a cover-to-join representation of $E$ in ${\mathrm B}_{tight}(E)$. 

\subsection{Stone duality for Boolean algebras} 
Let $B$ be a Boolean algebra. A character $\varphi\colon B \to \{0,1\}$ is called an {\em ultra character}, if $\varphi^{-1}(1)$ is an ultrafilter, or equivalently, a maximal filter in $B$, see, e.g., \cite{GH} for more details on this. 
The set of ultra characters coincides with the set of tight characters $\widehat{B}_{tight}$.

Suppose that $\psi\colon C\to D$ is a morphism between Boolean algebras. Then for any ultra character $\varphi\colon D\to \{0,1\}$ the map $\varphi\psi\colon  C\to \{0,1\}$ is an ultra character of $C$, which induces a continuous proper map $\widehat{D}_{tight}\to \widehat{C}_{tight}$.

Let $X$ be a  Stone space. Then the set of all its compact-opens forms a Boolean algebra, denote it by $\widehat{X}$. 
If $\alpha\colon X \to Y$ is a continuous proper map between Stone spaces then for any compact-open $A\subseteq Y$ its inverse image $\alpha^{-1}(A)$ is a compact-open in $X$. This induces a morphism $\widehat{Y}\to \widehat{X}$ of Boolean algebras.

It is a classical result stemming back to Stone \cite{Stone36} (see also \cite{Doctor}) that the assignments describedÊ give rise to functors which establish a dual equivalence between  the category of Boolean algebras and the category of Stone spaces.

\subsection{Boolean inverse semigroups} \label{subs:distr_bool}

\begin{definition} An inverse semigroup $S$  is called 
{\em Boolean} if $E(S)$ admits the structure of a Boolean algebra and for any $a,b\in S$ such that $a\sim b$ the least upper bound $a\vee b$, with respect to the natural partial order, exists in $S$.
\end{definition}

\begin{definition}
A {\em morphism} $\varphi\colon S\to T$ between  Boolean inverse semigroups is a morphism of semigroups such that $\varphi|_{E(S)}$ is a morphism of Boolean algebras and satisfying $\varphi(a\vee b) = \varphi(a) \vee \varphi(b)$ for any $a,b\in S$ such that $a\sim b$.
\end{definition}

Morphsims between Boolean inverse semigroups are sometimes referred to as {\em additive morphisms}.

A morphism $\varphi\colon S\to T$ between Boolean inverse semigroups is {\em weakly meet-preserving}, if for any $d\in T$ and $a,b\in S$ such that $d \leq \varphi(a), \varphi(b)$ there exists $c\in S$ such that $c\leq a,b$ and $d\leq \varphi(c)$.

\begin{definition} A {\em representation} of an inverse semigroup $S$ in a  Boolean inverse semigroup $T$ is a morphism of semigroups $\varphi\colon S\to T$ such that $\varphi|_{E(S)}\colon E(S) \to E(T)$ is a proper representation of the semilattice $E(S)$ in the Boolean algebra $E(T)$.
\end{definition}

A representation $\varphi\colon S\to T$ of an inverse semigroup $S$ in a Boolean inverse semigroup $T$ is called {\em proper}, if for any $t\in T$ there is $n\geq 1$ and $t_1,\dots, t_n\in T, s_1,\dots,s_n\in S$ such that $t=\bigvee_{i=1}^n t_i$ and $\varphi(s_i)\geq t_i$ for all $i=1,\dots, n$. 

\subsection{Stone groupoids} 

An introduction to \'etale groupoids can be found, e.g., in \cite{Renault, Paterson}.  If ${\mathcal{G}}$ is a groupoid, the unit space of ${\mathcal G}$ is denoted by ${\mathcal G}^{(0)}$, the set of composable pairs by ${\mathcal G}^{(2)}$, and the domain and range maps - by $d$ and $r$, respectively. 
A subset $A\subseteq {\mathcal G}$ is called a {\em local bisection} if the restrictions of the maps $d$ and $r$ to $A$ are injective.

A {\em topological groupoid} is a groupoid ${\mathcal G}$ with a topology with respect to which both the multiplication and the inversion maps are continuous. A topological groupoid ${\mathcal G}$  is {\em \'etale} if the map $d$ is a local homeomorphism. Because $d(x) = r(x^{-1})$ for all $x\in {\mathcal G}$, it follows that, in an \'etale groupoid, both of the maps $d$ and  $r$ are local homeomorphisms.

An \'etale groupoid ${\mathcal G}$ is called  a {\em Stone groupoid} if its unit space ${\mathcal G}^{(0)}$ is a  Stone space. 

The notion of a morphism between Stone groupoids that we will need is that of a relational covering morphism \cite{KL17} which we now recall.  

Let ${\mathcal G}_1$ and ${\mathcal G}_2$ be \'etale topological groupoids and  ${\mathcal P}({\mathcal G}_2)$ be the powerset of ${\mathcal G}_2$.
A {\em relational covering morphism} from ${\mathcal G}_1$ to ${\mathcal G}_2$ is a pair $(f,g)$ where $g\colon  {\mathcal G}_1^{(0)} \to {\mathcal G}_2^{(0)}$ is a proper continuous map and $f\colon {\mathcal G}_1 \to {\mathcal P}({\mathcal G}_2)$ is a function such that the following axioms hold:
\begin{enumerate}[(RM1)]
\item if  $y\in f(x)$ where $x\in {\mathcal G}_1$ then $d(y)=gd(x)$ and $r(y)=gr(x)$;
\item if $(x,y)\in {\mathcal G}_1^{(2)}$ and $(s,t)\in {\mathcal G}_2^{(2)}$ are such that $s\in f(x)$ and $t\in f(y)$ then $st\in f(xy)$;
\item if $d(x)=d(y)$ (or $r(x)=r(y)$) where $x,y\in {\mathcal G}_1$ and $f(x)\cap f(y)\neq \varnothing$ then $x=y$;
\item if $y=g(x)$ and $d(t)=y$ (resp. $\!\!\!$ $r(t)=y$) where $x\in {\mathcal G}_1^{(0)}$ and $t\in {\mathcal G}_2$ then there is $s\in {\mathcal G}_1$ such that $d(s)=x$ (resp. $\!\!\!$ $r(s)=x$) and $t\in f(s)$;
\item for any open set $A\subseteq {\mathcal G}_2$: $f^{-1}(A)=\{x\in {\mathcal G}_1\colon f(x)\cap A\neq \varnothing\}$ is an open set in ${\mathcal G}_1$;
\item  for any $t\in {\mathcal G}_1^{(0)}: g(t)\in f(t)$;
\item for any $x\in {\mathcal G}_1: f(x^{-1}) = (f(x))^{-1}$.
\end{enumerate}

Axioms (RM1) - (RM6) are taken from \cite[7.2]{KL17}, and the axiom (RM7) from \cite[7.4]{KL17}.
Axiom (RM2) can be looked at as a weak form of preservation of multiplication; (RM3) tells us that $f$ is {\em star-injective} and (RM4) that it is {\em  star-surjective}; (RM5) tells us that $f$ is a {\em lower-semicontinuous relation}. 
Note that a relational covering morphism $(f,g)$ is entirely determined by $f$ and we have the equality $g(a)=df(a)$ for any $a\in {\mathcal G}_1^{(0)}$, see \cite[7.2]{KL17}. A relational covering morphism $(f,g)$ can thus be denoted simply by $f$. 

\begin{definition} A relational covering morphism $f\colon {\mathcal G}_1 \to {\mathcal P}({\mathcal G}_2)$ is called 
\begin{itemize}
\item {\em at least single valued}, if $|f(a)|\geq 1$ for all $a\in {\mathcal G}_1$;
\item   {\em at most single valued} if $|f(a)|\leq 1$ for all $a\in {\mathcal G}_1$;
\item {\em single valued} or a {\em covering functor} if $|f(a)|=1$ for all $a\in {\mathcal G}_1$.
\end{itemize}
\end{definition} 

\subsection{The universal groupoid of an inverse semigroup}\label{subs:universal} Let $S$ be an inverse semigroup. There is an action of $S$ on $\widehat{E(S)}$, called the {\em natural action}, defined as follows. The domain of action of $s\in S$ is the set $\{\varphi\in \widehat{E(S)}\colon \varphi({\bf d}(s))=1\}$, and for $\varphi$ from this domain we have $(s\cdot \varphi) (e) = \varphi(s^{-1}es)$. Let $\varphi \in \widehat{E(S)}$ and $s,t\in S$ be such that $
\varphi({\bf d}(s)) = \varphi({\bf d}(t)) =1$. We define $(s,\varphi) \sim (t,\varphi)$ if and only if there is $e \in E(S)$ such that $\varphi(e)=1$ and $se=te$. If $(s,\varphi) \sim (t,\varphi)$ we say that $s$ and $t$ define the same {\em germ} of the natural action of $S$ on $\widehat{E(S)}$.  We look at the germ $[s,\varphi]$ as at an arrow from $\varphi$ to $s\cdot \varphi$. This leads to the {\em groupoid of germs} ${\mathcal G}(S)$ of the natural action of $S$ on $\widehat{E(S)}$. The {\em patch topology} on ${\mathcal G}(S)$ has a basis consisting of the sets 
$$
\Theta[s] = \{[s,\varphi]\in {\mathcal G}(S)\colon \varphi({\bf d}(s))=1\}
$$
where $s\in S$ and
$$
\Theta[s; s_1,\dots s_n] = \{[s,\varphi]\in {\mathcal G}(S)\colon \varphi({\bf d}(s))=1, \varphi({\bf d}(s_1))=\dots = \varphi({\bf d}(s_n))=0\},
$$
where $n\geq 1$, $s, s_1,\dots, s_n \in S$ and $s_1,\dots, s_n\leq s$. Endowed with the patch topology, ${\mathcal{G}}(S)$ is a Stone groupoid known as {\em Paterson's  universal groupoid} of $S$ \cite{Paterson, Exel:comb}. The set of all compact-open bisections of ${\mathcal{G}}(S)$ forms a Boolean inverse semigroup, ${\mathrm B}(S)$, called the {\em Booleanization} of~$S$. The map $\iota_{{\mathrm B}(S)}\colon S\to {\mathrm B}(S)$, $s\mapsto \Theta[s]$, is a proper representation of $S$ in ${\mathrm B}(S)$. Moreover, every element of ${\mathrm B}(S)$ is a compatible join of the elements of the form $\iota_{{\mathrm B}(S)}(s)e$ where $s\in S$ and $e\in {\mathrm B}(E(S))$. 

\subsection{The tight groupoid of an inverse semigroup} Let $S$ be an inverse semigroup. 

\begin{definition} A representation $\varphi$ of an inverse semigroup $S$ in a 
Boolean inverse semigroup $T$ is called {\em cover-to-join} if $\varphi|_{E(S)}\colon E(S) \to E(T)$ is a cover-to-join representation of the semilattice $E(S)$ in the Boolean algebra $E(T)$.
\end{definition}

\begin{definition} A subset ${\mathcal X}\subseteq \widehat{E(S)}$ is {\em invariant} under the natural action of $S$ if $s\cdot \varphi \in {\mathcal X}$ for every $s\in S$ and $\varphi\in {\mathcal X}$ such that $\varphi({\bf d}(s))=1$.
\end{definition} 

By  virtue of \cite[Proposition 12.8]{Exel:comb}, $\widehat{E(S)}_{tight}$ is an invariant subset of $\widehat{E(S)}$ under the natural action of $S$. In addition, it is a closed subset of  $\widehat{E(S)}$. One can thus restrict the natural action of $S$ on $\widehat{E(S)}$ to $\widehat{E(S)}_{tight}$. The groupoid of germs of this restricted action is denoted by ${\mathcal G}_{tight}(S)$. It is known as {\em Exel's  tight groupoid of $S$}. The set of all compact-open bisections of ${\mathcal{G}}_{tight}(S)$ forms a Boolean inverse semigroup denoted ${\mathrm B}_{tight}(S)$ and called the {\em tight Booleanization} of $S$. The map $\iota_{{\mathrm B}_{tight}(S)}\colon S\to {\mathrm B}_{tight}(S)$, $s\mapsto \Theta[s]\cap {\mathcal G}_{tight}(S)$, is a proper tight representation of $S$ in ${\mathrm B}_{tight}(S)$.

\subsection{Stone duality for Boolean inverse semigroups}\label{subs:duality_distr_inverse}

If $S$ is a Boolean inverse semigroup, the map $\iota_{{\mathrm B}(S)}$ is  a natural isomorphism between $S$ and ${\mathrm B}(S)$.
Given a Stone groupoid ${\mathcal G}$, the set of its compact-open bisections forms a Boolean inverse semigroup, denoted ${\sf S}({\mathcal G})$, whose Stone groupoid is naturally isomorphic to ${\mathcal G}$.

Let $\psi\colon S\to T$ be a morphism between Boolean inverse semigroups, and let $[t,\varphi]\in {\mathcal G}_{tight}(T)$. Then $(\varphi\psi)|_{E(S)}\in \widehat{E(S)}_{tight}$ and we put 
$\overline{\psi}[t,\varphi] = \{[s, (\varphi\psi)|_{E(S)}]\colon \psi(s)=t\}$.
This defines a relational covering morphism from ${\mathcal{G}}_{tight}(T)$ to ${\mathcal{G}}_{tight}(S)$. 

In the reverse direction, if $f\colon {\mathcal G}_1\to {\mathcal P}({\mathcal G}_2)$ is a relational covering morphism between Stone groupoids, for any $U\in {\mathrm S}({\mathcal G}_2)$ we put $\overline{f}(U) = \{[s,\varphi]\in {\mathcal G}_1\colon f([s,\varphi])\cap U\neq \varnothing\}$ which is a compact-open local bisection of ${\mathcal G}_1$. This defines a morphism ${\mathrm S}({\mathcal G}_2)\to {\mathrm S}({\mathcal G}_1)$.

The  assignments described give rise to a pair of functors which establish a duality between the category of Boolean inverse semigroups and the category of Stone groupoids. Varying morphisms, one in fact obtains several dualities. To formulate them, we define morphisms of type 1, 2, 3 and 4 on the algebraic and on the topological side as it is presented in the  table below.

\vspace{0.3cm}

\begin{tabular}{|c|c|c|}
\hline
 &  morphisms between semigroups & morphisms between groupoids  \\
 \hline
type 1 &  morphisms & relational covering morphisms    \\ \hline
type 2&     proper moprhisms  & at least single valued  \\
& & relational covering morphisms \\ \hline
type 3& weakly meet-preserving & at most single valued  \\
&  moprhisms & relational covering morphisms \\ \hline
type 4 & proper and weakly meet preserving & continuous covering functors \\
& morphisms & \\ \hline
\end{tabular}
\vspace{0.3cm}

The following result \cite[Theorem 8.20]{KL17} upgrades morphisms in earlier versions due to Lawson \cite[Theorem 1.4]{Lawson12} and Lawson and Lenz \cite[Theorem 3.25]{LL13}.

\begin{theorem}\label{th:dual_distr} 
For each $k=1, 2, 3, 4$ the category of Boolean inverse semigroups with morphisms of type $k$ is dually equivalent to the category of Stone groupoids with morphisms of type $k$.
\end{theorem}

\section{$X$-to-join representations of semilattices and inverse semigroups}\label{sec:X}

\subsection{$X$-to-join representations of semilattices}\label{subs:X_semilattices}
For a set $E$ by ${\mathcal P}_{fin}(E)$ we denote the set of all finite subsets of $E$.

\begin{definition}  Let $E$ be a semilattice, $B$  a Boolean algebra and $X\subseteq E\times {\mathcal P}_{fin}(E)$. 
 A representation $\varphi\colon E\to B$ will be called  {\em $X$-to-join}, if it is proper and
\begin{equation}\label{eq:j28a}
\varphi(e) = \bigvee_{i=1}^n \varphi(e_i)
\end{equation}
for all $(e, \{e_1,\dots, e_n\})\in X$.
\end{definition}

We denote the set of all $X$-to-join characters of $E$ by  $\widehat{E}_X$.
A standard argument (to be found, e.g., in \cite[Lemma~34.1]{GH})  ensures that $\widehat{E}_X$ is a closed subset of $\widehat{E}$. In the subspace topology, it is a Stone space. 
Consequently, the set of compact-opens of the space  $\widehat{E}_{X}$ forms a Boolean algebra which we denote by ${\mathrm B}_{X}(E)$ and call the {\em $X$-to-join Booleanization} of $E$. The map $\iota_{{\mathrm B}_{X}(E)}\colon E\to {\mathrm B}_{X}(E)$ given by $a\mapsto M_a\cap \widehat{E}_{X}$ is an $X$-to-join representation of $E$ in ${\mathrm B}_{X}(E)$ and, moreover,  $\iota_{{\mathrm B}_{X}(E)}(E)$ generates ${\mathrm B}_{X}(E)$ as a Boolean algebra. The space $\widehat{E}_{X}$ is the dual space of the Boolean algebra ${\mathrm B}_{X}(E)$ by means of the Stone duality.

\begin{example}\mbox{} 
\begin{enumerate}
\item If $X=\varnothing$ then an $\varnothing$-to-join representation of $E$ is obviously just a proper representation of $E$.
\item If $X$ consists of all $(e,\{e_1,\dots, e_n\})$ where $\{e_1,\dots, e_n\}$ is a cover of $e$, $X$-to-join representations coincide with  cover-to-join representations, and thus with proper tight representations.
\end{enumerate}
\end{example}

The following result provides a universal property of ${\mathrm B}_X(E)$. 

\begin{theorem} \label{th:booleanization1} Let $E$ be a semilattice, $B$ a Boolean algebra and $\varphi\colon E\to B$  an $X$-to-join representation. Then there is a unique morphism of Boolean algebras $\psi\colon {\mathrm B}_X(E)\to B$ such that $\varphi=\psi\iota_{{\mathrm B}_X(E)}$.
\end{theorem}

\begin{proof} Let $\widehat{B} = \widehat{B}_{tight}$ be the dual Stone space of the Boolean algebra $B$. Then $B$ is isomorphic to the Boolean algebra of compact-opens of $\widehat{B}$ via the map $b\mapsto M_b \cap \widehat{B}$. Applying the Stone duality, the existence of $\psi$ is equivalent to the  existence of a continuous proper map $\gamma \colon \widehat{B}\to \widehat{E}_{X}$ such that $\varphi=\gamma^{-1}\iota_{B_X(E)}$. Let $\alpha \in \widehat{B}$. Then $\alpha\varphi \in \widehat{E}_X$ (note that $\alpha\varphi$ is non-zero because $\varphi$ is proper) and we put $\gamma(\alpha) = \alpha\varphi$. We show that $\gamma$ is a proper continuous map, that is, $\gamma^{-1}$ takes compact-opens to compact-opens. It is enough to restrict attention to basic compact-opens. For any $\alpha \in \widehat{B}$ and $a, b_1, \dots, b_k\in E$ (where $k\geq 0$) such that $b_1,\dots, b_k\leq a$ we have: 
\begin{align*}
\alpha \in \gamma^{-1}(M_{a;b_1,\dots, b_k}\cap \widehat{E}_X) & \Leftrightarrow   \gamma(\alpha)\in M_{a;b_1,\dots, b_k}\cap \widehat{E}_X   \\
& \Leftrightarrow  \alpha\varphi \in M_{a;b_1,\dots, b_k}  \cap \widehat{E}_X\\
& \Leftrightarrow   \alpha\varphi(a) = 1, \alpha\varphi(b_1) = \dots =  \alpha\varphi(b_k)= 0\\
& \Leftrightarrow   \alpha \in M_{\varphi(a);\varphi(b_1), \dots, \varphi(b_k)} \cap \widehat{B}.
\end{align*}
It follows that $\gamma^{-1}(M_{a;b_1,\dots, b_k}\cap \widehat{E}_X) = M_{\varphi(a)\setminus (\varphi(b_1)\vee\dots \vee\varphi(b_k))} \cap \widehat{B}$, which is a compact-open. Therefore, the required morphism of Boolean algebras $\psi\colon {\mathrm B}_X(E)\to B$ exists and satisfies $\psi(M_{a;b_1,\dots, b_k}\cap \widehat{E}_X) = M_{\varphi(a)\setminus (\varphi(b_1)\vee\dots \vee\varphi(b_k))}  \cap \widehat{B}$ (here $B$ is identified with the Boolean algebra of compact-opens of $\widehat{B}$). That $\psi$ is unique is easy to show using the fact that the image of $\iota_{{\mathrm B}_X(E)}$ generates ${\mathrm B}_X(E)$.
\end{proof}

\subsection{Connection with $\pi$-tight representations}\label{subs:pitight}

Let $\pi\colon E\to B$ be a representation of a semilattice $E$ in a Boolean algebra $B$. Define $X_{\pi}$ as the set consisting of all $(e, \{e_1,\dots, e_n\}) \in E\times {\mathcal P}_{fin}(E)$ such that $\pi(e) = \bigvee_{i=1}^n \pi(e_i)$. Then $X_{\pi}$-to-join representations of $E$ coincide with $\pi$-tight representations considered in~\cite{ES18}.
The next result was proved in  \cite[Theorem 15.11]{ES18}.

\begin{lemma}\label{lem:character}
Let $\pi\colon E\to B$ be a representation of a semilattice $E$ in a Boolean algebra $B$ such that $\pi(E)$ generates $B$ and let $\varphi$ be an $X_{\pi}$-to-join character of $E$. Then
there is an ultra character, $\psi$, of $B$ such that $\varphi = \psi\pi$.
\end{lemma}

The following is a reformulation of a result from \cite{ES18}. Its proof follows from the proof of \cite[Theorem 15.11]{ES18}. We provide another proof.

\begin{theorem} \label{th:isom} Let $\pi\colon E\to B$ be a representation of a semilattice $E$ in a 
 Boolean algebra $B$ such that $\pi(E)$ generates $B$ as a  Boolean algebra. Then $B$ is isomorphic to ${\mathrm B}_{X_{\pi}}(E)$.
\end{theorem}

\begin{proof} Since $\pi$ is $X_{\pi}$-to-join, Theorem \ref{th:booleanization1} implies that there is a morphism of Boolean algebras $\psi\colon {\mathrm B}_{X_{\pi}}(E)\to B$ such that $\pi=\psi\iota_{{\mathrm B}_{X_{\pi}}(E)}$. Moreover, since $\pi(E)$ generates $B$, we have that $\psi$ is surjective. Thus, the Stone duality yields that its dual proper continuous map $\widehat{B}_{tight} \to \widehat{E}_{X_{\pi}}$ given by $\alpha\mapsto \alpha\pi$ is injective. Lemma \ref{lem:character} implies that this map is in fact bijective, which, invoking the Stone duality, shows that $\psi$ is an isomorphism. 
\end{proof}

\subsection{$X$-to-join representations of inverse semigroups}

Let $S$ be an inverse semigroup and $T$ a Boolean inverse semigroup.

\begin{definition} Let $X\subseteq E(S)\times {\mathcal P}_{fin}(E(S))$.  A representation $\varphi\colon S\to T$ will be called {\em $X$-to-join} if $\varphi|_{E(S)}\colon E(S)\to E(T)$ is an $X$-to-join representation of the semilattice $E(S)$ in the Boolean algebra $E(T)$. 
\end{definition}

\begin{definition} Let $X\subseteq E(S)\times {\mathcal P}_{fin}(E(S))$. We will say that $X$ is {\em $S$-invariant}, or simply {\em invariant}, if $(e, \{e_1,\dots, e_n\})\in X$ implies that $(s^{-1}es, \{s^{-1}e_1s,\dots, s^{-1}e_ns\})\in X$, for all $s\in S$. 
\end{definition}

By $X'$ we denote the smallest $S$-invariant subset of $E(S)\times {\mathcal P}_{fin}(E(S))$ that contains $X$. The set $X'$ exists because $E(S)\times {\mathcal P}_{fin}(E(S))$ is $S$-invariant, and the intersection of any family of $S$-invariant subsets of $E(S)\times {\mathcal P}_{fin}(E(S))$ is $S$-invariant as well. It follows that $X'$ is the intersection of all $S$-invariant subsets of $E(S)\times {\mathcal P}_{fin}(E(S))$ which contain $X$.

\begin{lemma} Let $S$ be an inverse semigroup, $X\subseteq E(S)\times {\mathcal P}_{fin}(E(S))$ and $T$ a Boolean inverse semigroup. A representation $\varphi\colon S\to T$ is an $X$-to-join representation if and only if it is an $X'$-to-join representation.
\end{lemma}

\begin{proof} Let $(e, \{e_1,\dots, e_n\})\in X$ and $s\in S$. It  suffices to show that $\varphi(s^{-1}es) = \bigvee_{i=1}^n \varphi(s^{-1}e_is)$. This follows applying the fact that in a Boolean inverse semigroup the multiplication distributes over finite compatible joins: $\varphi(s^{-1}es)$ is equal to
\begin{equation*}
\varphi(s^{-1})\varphi(e)\varphi(s) = \varphi(s^{-1})\left(\bigvee_{i=1}^n(e_i)\right)\varphi(s) = \\ 
\bigvee_{i=1}^n\varphi(s^{-1})\varphi(e_i)\varphi(s) = \bigvee_{i=1}^n \varphi(s^{-1}e_is).
\end{equation*}
\end{proof}

The following is immediate from the definitions.

\begin{lemma} \label{lem:invariant} Let $X\subseteq E(S)\times {\mathcal P}_{fin}(E(S))$. Then $\widehat{E(S)}_{X'}$ is a closed and invariant subset of $\widehat{E(S)}$ under the natural action of $S$.
\end{lemma}

\begin{example} \label{ex:invariant} Let $X$ be the set that defines cover-to-join representations, that is, it consists of all  $(e, \{e_1,\dots, e_n\})\subseteq E(S) \times {\mathcal P}_{fin}(E(S))$ where $\{e_1,\dots, e_n\}$ is a cover of $e$. We show that $X$ is invariant.
Suppose that $s\in S$, $\{e_1,\dots, e_n\}$ is a cover of $e$ and show that $\{s^{-1}e_1s,\dots, s^{-1}e_ns\}$ is a cover for $s^{-1}es$. First, we may assume that $e\leq ss^{-1}$, otherwise we replace $s$ by $t=ss^{-1}e$ and have that $s^{-1}es = t^{-1}et$ and $s^{-1}e_is = t^{-1}e_it$ for all $i$. Let now $f\leq s^{-1}es$, $f\neq 0$. Then
$sfs^{-1} \leq ss^{-1}ess^{-1} = e$, $sfs^{-1}\neq 0$. Since $\{e_1,\dots, e_n\}$ is a cover of $e$, it follows that there is $i$ such that $e_isfs^{-1} = g \neq 0$. But then $s^{-1}e_isfs^{-1}s = s^{-1}gs \neq 0$.
That is, $s^{-1}e_isf \neq 0$. It follows that  $\{s^{-1}e_1s,\dots, s^{-1}e_ns\}$ is a cover for $s^{-1}es$. 
\end{example}

If $X\subseteq  E(S)\times {\mathcal P}_{fin}(E(S))$, one can restrict the natural action of $S$ on $\widehat{E(S)}$ to the closed invariant subset $\widehat{E(S)}_{X'}$. The groupoid of germs of this restricted action is denoted by ${\mathcal G}_{X}(S)$. 
We will denote the dual Boolean inverse semigroup ${\mathsf S} ({\mathcal G}_{X}(S))$  by ${\mathrm B}_X(S)$ and call it the {\em $X$-to-join-Booleanization} of $S$. The map $\iota_{{\mathrm B}_{X}(S)}\colon S\to {\mathrm B}_{X}(S)$ given by $s\mapsto \Theta[s]\cap {\mathcal G}_X(S)$ is an $X$-to-join representation of $S$ in ${\mathrm B}_{X}(S)$. The groupoid ${\mathcal G}_X(S)$ is the dual  groupoid of the the Boolean inverse semigroup  ${\mathrm B}_{X}(S)$ by means of the extended Stone duality (see Theorem~\ref{th:dual_distr}).

\begin{example}\mbox{}
\begin{enumerate}
\item 
If $X=\varnothing$, we have ${\mathcal G}_{X}(S)={\mathcal G}(S)$ and ${\mathrm B}_{\varnothing}(S)= {\mathrm B}(S)$.

\item In the case where $X$ defines cover-to-join representations, we have ${\mathcal G}_{X}(S)={\mathcal G}_{tight}(S)$ and  ${\mathrm B}_X(S)={\mathrm B}_{tight}(S)$.
\end{enumerate}
\end{example}

\begin{theorem} \label{th:booleanization2} Let $S$ be an inverse semigroup and $X\subseteq E(S)\times {\mathcal P}_{fin}(E(S))$. 
 Let further $B$ be a Boolean inverse semigroup and $\varphi\colon S\to B$ an $X$-to-join representation (resp. $\!\!\!$ a proper $X$-to-join representation). Then there is a unique morphism (resp. $\!\!\!$ a unique proper morphism) of Boolean inverse semigroups $\psi\colon {\mathrm B}_X(S)\to B$ such that $\varphi=\psi\iota_{{\mathrm B}_X(S)}$.
\end{theorem}

\begin{proof} Supose that $\varphi\colon S\to B$ is an $X$-to-join representation. Applying Theorem \ref{th:dual_distr}, the existence of a required morphism  $\psi\colon {\mathrm B}_X(S)\to B$ is equivalent to the existence of a relational covering morphism $\gamma\colon {\mathcal G}_{tight}(B) \to {\mathcal P}({\mathcal G}_X(S))$ such that $\varphi=\gamma^{-1}\iota_{{\mathrm B}_X(S)}$. Let $[b, \alpha]\in {\mathcal G}_{tight}(B)$.  We have $\alpha\varphi|_{E(S)} \in \widehat{E(S)}_{X'}$. Note also that $\varphi(s) =b$ implies that \mbox{$\alpha\varphi({\bf d}(s))=\alpha({\bf d}(b))$.}  
We can thus define
\begin{equation}\label{eq:29}
\gamma([b,\alpha]) = \{[s, \alpha\varphi|_{E(S)}] \colon \varphi(s)=b\}.
\end{equation}
We show that $\gamma$ is a relational covering morphism from  ${\mathcal G}_{tight}(B)$ to ${\mathcal G}_X(S)$. For $s, s_1, \dots, s_n \in S$ where $n\geq 0$ and $s_i\leq n$ for all $i$ we have:
\begin{align*}
[b, \alpha]\in \,  \gamma^{-1}(\Theta[s; s_1,\dots, s_n]\cap {\mathcal G}_X(S))  & \Leftrightarrow  \gamma([b,\alpha])\cap \Theta[s; s_1,\dots, s_n]\cap {\mathcal G}_X(S) \neq \varnothing   \\
& \Leftrightarrow b=\varphi(s), \alpha\varphi({\bf d}(s))=1, \alpha\varphi({\bf d}(s_1))=\alpha\varphi({\bf d}(s_n))=0  \\
& \Leftrightarrow [b,\alpha] \in \Theta[\varphi(s); \varphi(s_1),\dots, \varphi(s_n)] \cap {\mathcal G}_{tight}(B).\\
\end{align*} It follows that 
\begin{equation}\label{eq:aux:n15}
\gamma^{-1}(\Theta[s; s_1,\dots, s_n]|_{\widehat{E(S)}_{X'}}) = \Theta[\varphi(s)\setminus (\varphi(s_1)\vee \dots \vee \varphi(s_n))] |_{\widehat{E(B)}_{tight}}.
\end{equation}
Therefore, $\varphi=\gamma^{-1}\iota_{{\mathrm B}_X(S)}$ holds and $\gamma$ restricted to the unit space of  ${\mathcal G}_{tight}(B)$ is a proper and continuous map.  Let us show that the axioms (RM1) -- (RM7) hold. 

(RM1) Let $[s, \alpha\varphi|_{E(S)}] \in \gamma([b,\alpha])$  where $\varphi(s) = b$. Then $\varphi({\mathbf d}(s))={\mathbf d}(b)$ so that we have that $[{\mathbf d}(s), \alpha\varphi|_{E(S)}] \in \gamma([{\mathbf d}(b),\alpha])$.

(RM2) Suppose that $[s, \alpha\varphi|_{E(S)}] \in \gamma([b,\alpha])$, $[t, \beta\varphi|_{E(S)}] \in \gamma([c,\beta])$  where $([b,\alpha], [c,\beta]) \in {\mathcal G}_{tight}(B)^{(2)}$, $([s, \alpha\varphi|_{E(S)}], [t, \beta\varphi|_{E(S)}]) \in {\mathcal G}_X(S)^{(2)}$. Then $\varphi(s) = b$, $\varphi(t) = c$ and we have $[b,\alpha][c,\beta] = [bc, \beta]$, $[s, \alpha\varphi|_{E(S)}][t, \beta\varphi|_{E(S)}] = [st, \beta\varphi|_{E(S)}]$. Since $\varphi(st) = bc$, it follows that $[st, \beta\varphi|_{E(S)}]\in \gamma([bc, \beta])$.

(RM3) Note that $d([b,\alpha])=d([c,\beta])$ if and only if $\alpha=\beta$. Since $\gamma([b,\alpha])$ consists of all $[s, \alpha\varphi|_{E(S)}]$ where $\varphi(s)=b$ and $\gamma([c,\alpha])$ consists of all $[s, \alpha\varphi|_{E(S)}]$ where $\varphi(s)=c$ and in view of (RM7), axiom (RM3) follows.

(RM4) Suppose that $\beta \in \gamma(\alpha)$ where $\alpha \in \widehat{E(B)}_{tight}$ and $\beta \in \widehat{E(S)}_{X'}$. This means that $\beta = \alpha\varphi|_{E(S)}$. For any  $[s, \beta]$ we need to show that there is $[t,\alpha]$ such that $[s,\beta] \in \gamma([t,\alpha])$. Note that $[s, \beta] \in \gamma([t,\alpha])$ if and only if $\beta = \alpha\varphi|_{E(S)}$ and $t = \varphi(s)$. Therefore, it is enough to take $t=\varphi(s)$. In view of (RM7), axiom (RM4) follows.

(RM5) This follows from \eqref{eq:aux:n15}.

(RM6) Let $\alpha \in {\mathcal G}_{tight}(B)^{(0)} = \widehat{E(B)}_{tight}$. Then $\alpha$ is identified with the germ $[e, \alpha]$ where $\alpha(e)=1$. We have that $\gamma([e, \alpha])$ is the set of all $[s, \alpha\varphi|_{E(S)}]$ where $\varphi(s)=e$. 
Since the restriction of $\varphi$ to $E(S)$ is proper, there are $e_1, \dots, e_n \in E(S)$ such that $\varphi(e_1)\vee \dots \vee \varphi(e_n) \geq e$. Since $\alpha(e)=1$, there is $i$ such that $\alpha(\varphi(e_i))=1$. But then $[e, \alpha] = [\varphi(e_i), \alpha]$. It follows that $[e_i, \alpha\varphi|_{E(S)}] \in \gamma([e, \alpha])$. 

(RM7) Let $[b,\alpha]\in {\mathcal G}_{tight}(B)$. Then $[b,\alpha]^{-1} = [b^{-1}, \beta]$ where $\beta = b\cdot \alpha$. That is $\beta(e)= \alpha (b^{-1}eb)$ for all $e\in \widehat{E(B)}_{tight}$. By the definition of $\gamma$, we have that 
$\gamma([b,\alpha])$ consists of all $[s, \alpha\varphi|_{E(S)}]$ where $\varphi(s) = b$. Hence $\gamma([b,\alpha])^{-1}$ consists of all $[s^{-1}, (b\cdot \alpha)\varphi|_{E(S)}]$. On the other hand, we have that $\gamma([b,\alpha]^{-1}) = \gamma([b^{-1}, b\cdot \alpha])$ consists of all $[t, (b^{-1}\cdot b \cdot \alpha)\varphi|_{E(S)}]$ where $\varphi(t) = b^{-1}$. The latter equality is equivalent to $\varphi(t^{-1}) = b$. Also, $(b^{-1}\cdot b\cdot \alpha)(e) = \alpha (bb^{-1}ebb^{-1}) = \alpha(bb^{-1}) \alpha(e) = \alpha(e)$. The equality $\gamma([b,\alpha])^{-1} = \gamma([b,\alpha]^{-1})$ follows.

Suppose now that $\varphi\colon S\to B$ is a proper $X$-to-join representation. Due to Theorem~\ref{th:dual_distr}, it sufficies to check that that $\gamma$ is at least single-valued. Let $[b, \alpha]\in {\mathcal G}_{tight}(B)$. Since $\varphi$ is proper, there is $n\geq 1$ and $b_1,\dots, b_n\in B$ where $b=b_1\vee \dots\vee b_n$, $s_1,\dots, s_n\in S$ such that
$\varphi(s_i)\geq b_i$ for all $i=1,\dots, n$. Because $\alpha({\mathbf d}(b))=1$, there is $i\in \{1,\dots, n\}$ such that $\alpha({\mathbf d}(b_i))=1$. It follows that $\alpha({\mathbf d}(\varphi(s_i)))=1$. Therefore, $[b, \alpha] = [b_i, \alpha] = [\varphi(s_i), \alpha]$. Hence, every $[b, \alpha]\in {\mathcal G}_{tight}(B)$ is equal to some $[\varphi(s), \alpha]$ where $s\in S$, thus $\gamma([b,\alpha]) \neq \varnothing$. 

Uniqueness of $\psi$ easily follows from the uniqueness of the map $\psi$ in Theorem \ref{th:booleanization1} and the fact that every element of ${\mathrm B}_X(S)$ is a compatible join of elements of the form $\iota_{{\mathrm B}_X(S)}(s)e$ where $s\in S$ and $e\in E({\mathrm B}_X(S))$.
This finishes the proof.
\end{proof}

\section{Prime and core representations}\label{sect:examples}

The notion of a cover-to-join representation of a semilattice in a  Boolean algebra is an extension of the notion of a morphism between  Boolean algebras (see Proposition \ref{prop:gba_tight}).
We now consider a new class of representations of semilattices in Boolean algebras which in a similar way generalize proper morphisms from distributive lattices to Boolean algebras. 

Let $E$ be a semilattice. If $X\subseteq E$ is a finite subset and $x\in E$ is the least upper bound of $X$ we write $x=\bigvee X$ (of course, $\bigvee X$ does not necessarily exist).

\begin{definition} Let $E$ be a semilattice and $B$ a Boolean algebra. A proper representation $\varphi\colon E\to B$ will be called {\em prime}, provided that for any $x\in E$ and any finite cover $X$ of $x$ the following implication holds:
\begin{equation}\label{eq:def_prime}
\text{if } 
x=\bigvee X \text{ then } \varphi(x)=\bigvee_{x\in X} \varphi(x).
\end{equation}
\end{definition}

Clearly, any cover-to-join representation is prime. 

\begin{proposition}\label{prop:1} Let $B$ be a Boolean algebra and suppose that the semilattice $E$ admits the structure of a distributive lattice. Then a proper representation $\varphi\colon E\to B$ is prime if and only if it is a lattice morphism.
\end{proposition}

\begin{proof} Suppose that $\varphi$ is prime. Let $x,y\in E$. We show that $\{x,y\}$ is a cover of $x\vee y$. Let $0\neq z\leq x\vee y$. Then $z= z\wedge (x\vee y) = (z\wedge x)\vee (z\wedge y)$. If follows that either $z\wedge x$ or $z\wedge y$ is non-zero. Therefore, $\varphi(x\vee y) = \varphi(x)\vee \varphi(y)$, so that $\varphi$ is a lattice morphism. The reverse implication is immediate.
\end{proof}

Prime representations are $X$-to-join representations where $X$ consists of all $(e, \{e_1,\dots, e_n\})$ where $e=\bigvee_{i=1}^n e_n$ and $\{e_1,\dots, e_n\}$ is a cover of $e$. For prime representations, we denote  $\widehat{E(S)}_X$ by $\widehat{E(S)}_{prime}$.

\begin{example}\label{e:3} Let $n\geq 1$ and  $E_n=\{0, e_1,\dots, e_{n}\}$ with $0\leq e_1\leq \dots \leq e_{n}$. Since $E_n$ is a distributive lattice, $\iota_{{\mathrm B}_{prime}(E_n)}\colon E\to {\mathrm B}_{prime}(E_n)$ is an injective lattice morphism and ${\mathrm B}_{prime}(E_n)$ is isomorphic to the Booleanization $E_n^-$ of the distributive lattice $E_n$ (by Proposition \ref{prop:1} and the universal property of ${\mathrm B}_{prime}(E_n)$). Observe that  $\iota_{{\mathrm B}_{prime}(E_n)}$ is prime but not cover-to-join and that $\iota_{{\mathrm B}_{tight}(E_n)}$ maps $E_n$ onto a two-element Boolean algebra. 
\end{example}

Following Exel \cite{Exel09}, the element $f$ satisfying $f\leq e$ is said to be {\em dense} in $e$ provided that there is no non-zero element $d\leq e$ such that $d\wedge f=0$. It is clear that $f$ is dense in $e$ if and only if $\{f\}$ is a cover of $e$.

\begin{definition} Let $E$ be a semilattice and $B$ a Boolean algebra. A proper representation $\varphi\colon E\to B$ will be called {\em core}, provided that for any $e,f\in E$ such that $f\leq e$ and $f$ is dense in $e$ we have: $\varphi(f)=\varphi(e)$.
\end{definition}

Clearly, any cover-to-join representation of $E$ is core. 
Core representations are $X$-to-join representations where $X$ consists of all $(e, \{f\})$ with $f$ being dense in $e$. 
In the case of core representations, we denote $\widehat{E(S)}_X$ by $\widehat{E(S)}_{core}$. 

Let $S$ be an inverse semigroup and $T$ a Boolean inverse semigroup. A representation $\varphi\colon S\to T$ will be called {\em prime} (resp. {\em core}) provided that $\varphi|_{E(S)}\colon E(S)\to E(T)$ is a prime (resp. core) representation of the semilattice $E(S)$ in the Boolean algebra $E(T)$.

In a similar way as in Example \ref{ex:invariant}, one can show that the sets $X_1, X_2\subseteq E(S)\times {\mathcal P}_{fin}(E(S))$, where $X_1$ consists of all $(e, \{e_1,\dots, e_n\})$ where $e=\bigvee_{i=1}^n e_n$ and $\{e_1,\dots, e_n\}$ is a cover of $e$, and $X_2$ consists of all $(e, \{f\})$ with $f$ being dense in $e$, are $S$-invariant. Bearing in mind Lemma \ref{lem:invariant}, we obtain the following statement.

\begin{proposition} \mbox{} 
$\widehat{E(S)}_{prime}$ and $\widehat{E(S)}_{core}$ are invariant subsets of $\widehat{E(S)}$.
\end{proposition}

\section{Generators and relations} \label{sect:generators}
We first briefly describe the universal algebraic approach to Boolean inverse semigroups to be found in \cite{W17}. Let $S$ be a Boolean inverse semigroup. The operations $\setminus$ and $\vee$ on $E(S)$ can be extended to $S$ as follows: 
$$
s\setminus t = ({\bf r}(s)\setminus {\bf r}(t))s({\bf d}(s)\setminus {\bf d}(t)),
$$
$$
s\triangledown t = (s\setminus t) \vee t,
$$
where $\vee$ on the right hand side of the last formula is the join with respect to the natural partial order of the compatible elements $s\setminus t$ and $t$. The operation $\setminus$ is called the {\em difference} operation and $\triangledown$ is the (right-handed) {\em skew join operation}. Upon adding the operations $\setminus$ and $\triangledown$ to the signature of a Boolean inverse semigroup $S$, it becomes an algebra $(S; 0, ^{-1}, \cdot, \setminus, \triangledown)$. We say that the algebra $(S; 0, ^{-1}, \cdot, \setminus, \triangledown)$ is {\em attached} to $S$ (or, more presicely, to $(S; 0, ^{-1}, \cdot)$) and that $S$ is the {\em inverse semigroup reduct} of $(S; 0, ^{-1}, \cdot, \setminus, \triangledown)$. The operation $\triangledown$ is an extension of the partially defined operation $\vee$ to the whole $S\times S$. If the elements $a_1,\dots, a_n$ are pairwise compatible (for example, if $a_i\in E(S)$ for all $i$) then the element $\bigtriangledown_{i=1}^n a_i$ is well defined and equals  $\vee_{i=1}^n a_i$.

\begin{proposition}[\cite{W17}]An algebra $(S; 0, ^{-1}, \cdot, \setminus, \triangledown)$ of type $(0,1,2,2,2)$ is an algebra attached to a Boolean inverse semigroup, as described above, if and only if $(S; 0, ^{-1}, \cdot)$ is an inverse semigroup with zero $0$ and  the following identities hold:
\begin{enumerate}
\item[(1)] $({\mathbf d}(x) \setminus {\mathbf d}(y))^2  = {\mathbf d}(x) \setminus {\mathbf d}(y)$, $({\mathbf d}(x) \triangledown {\mathbf d}(y))^2  = {\mathbf d}(x) \triangledown {\mathbf d}(y)$; 
\item[(2)] all the defining identities (and hence all the identities) of the variety of Boolean algebras  if variables $x,y, \dots, $
are replaced by ${\mathbf d}(x)$, ${\mathbf d}(y)$, $\dots$, and the operations $0$, $\wedge$, $\vee$ and $\setminus$ (of Boolean algebras) are replaced by the operations $0$, $\cdot$, $\triangledown$ and $\setminus$ (of algebras attached to Boolean inverse semigroups), respectively;
\item[(3)] $x\triangledown y \geq x\setminus y$, $x\triangledown y \geq y$;
\item[(4)] ${\mathbf d}(x\triangledown y) = {\mathbf d}(x\setminus y) \triangledown {\mathbf d}(y)$;
\item[(5)] $x\setminus y = ({\bf r}(x)\setminus {\bf r}(y))x({\bf d}(x)\setminus {\bf d}(y))$;
\item[(6)] $z(({\bf d}(x)\setminus {\bf d}(y) \triangledown {\bf d}(y)) = z({\bf d}(x)\setminus {\bf d}(y)) \triangledown z {\bf d}(y)$.
\end{enumerate}
Moreover, if  $(S; 0, ^{-1}, \cdot, \setminus, \triangledown)$ is an algebra of type $(0,1,2,2,2)$ such that $(S; 0, ^{-1}, \cdot)$ is an inverse semigroup with zero $0$ and the above stated identities hold, then $(S; 0, ^{-1}, \cdot)$ is a Boolean inverse semigroup and $(S; 0, ^{-1}, \cdot, \setminus, \triangledown)$ is  the algebra attached to $(S; 0, ^{-1}, \cdot)$. \end{proposition}

It follows that Boolean inverse semigroups in the extended signature form a variety of algebras. It is shown in \cite{W17} that morphisms between Boolean inverse semigroups in the extended signature can be characterized precisely as additive morphisms between their inverse semigroup reducts. That is, a map $\varphi\colon S\to T$ between Boolean inverse semigroups is a morphisms between their attached algebras  if and only if:
\begin{enumerate}
\item $\varphi(st)=\varphi(s)\varphi(t)$ for all $s,t\in S$;
\item $\varphi(0)=0$;
\item $\varphi(a\vee b) = \varphi(a) \vee \varphi (b)$ for all $a,b\in S$ such that $a\sim b$.
\end{enumerate}

\begin{proposition} Let $X$ be a set and let $FI(X)$ be the free inverse semigroup on $X$. Then the free Boolean inverse semigroup (in the extended signature), $FBI(X)$, on $X$ is isomorphic to ${\mathrm B}(FI(X)\cup \{0\})$.
\end{proposition}

\begin{proof} Let $S$ be a Boolean inverse semigroup and let $\varphi\colon X\to S$ be a map. We need to show that $\varphi$ can be extended to an additive morphism $\widetilde{\varphi}\colon {\mathrm B}(FI(X)\cup \{0\}) \to S$. Firstly, the universal property of $FI(X)$ implies that $\varphi$ can be extended to a morphism of semigroups (without zero) $\varphi'\colon FI(X) \to S$. Attaching a zero element to $FI(X)$ and putting $\varphi'(0) = 0$, we get a morphism of inverse semigroup with zero $FI(X)\cup \{0\} \to S$. The claim now follows from the universal property of ${\mathrm B}(FI(X)\cup \{0\})$.
\end{proof}

We now describe the Boolean inverse semigroups ${\mathrm B}_X(S)$ (in the extended signature) by generators and relations.

\begin{proposition}\label{prop:relations} Let $S$ be an inverse semigroup and $X\subseteq E(S)\times {\mathcal P}_{fin}(E(S))$.  Then ${\mathrm B}_X(S)$ is generated by the set $\{[s]\colon s\in S\}$ subject to the relations:

\begin{enumerate} 
\item $[0]=0$; 
\item $[st] = [s][t]$ for all $s,t\in S$;
\item $[e] = \bigtriangledown_{i=1}^n [e_n]$ for all $(e, \{e_1,\dots, e_n\})\in X$.
\end{enumerate}
\end{proposition} 

\begin{proof} The statement follows from the universal property of ${\mathrm B}_X(S)$ (Theorem \ref{th:booleanization2}).
\end{proof}

\begin{corollary} 
\begin{enumerate}  Let $S$ be an inverse semigroup. 
\item The universal Booleanization ${\mathrm B}(S)$ is generated by the set $\{[s]\colon s\in S\}$ subject to the relations (1) and (2) from Proposition \ref{prop:relations}.
\item Let $X\subseteq E(S)\times {\mathcal P}_{fin}(E(S))$. Then ${\mathrm B}_X(S)$ is a quotient of ${\mathrm B}(S)$, that is obtained by adding relations (3) from Proposition \ref{prop:relations} to the defining relations of ${\mathrm B}(S)$.
\end{enumerate}
\end{corollary}

\begin{remark} In the case where the set $X$  defines cover-to-join representations, presentation of ${\mathrm B}_X(S)$ by generators and relations was treated in \cite{LV19}.
\end{remark}

\section{Weakly meet-preserving quotients of ${\mathrm B}(S)$ via $X$-to-join representations}\label{sect:all}

Let $S$ be an inverse semigroup and $X\subseteq E(S)\times {\mathcal P}_{fin}(E(S))$ be an invariant subset. The canonical quotient morphism of Boolean inverse semigroups ${\mathrm B}(S) \to {\mathrm B}_X(S)$ corresponds, my means of the duality of Theorem \ref{th:dual_distr}, to the inclusion map  ${\mathcal G}_X(S) \hookrightarrow {\mathcal G}(S)$ between their dual groupoids. Since the latter map is  single-valued, the canonical quotient morphism ${\mathrm B}(S) \to {\mathrm B}_X(S)$ is weakly meet-preserving.
We now show that any quotient of ${\mathrm B}(S)$ by a weakly meet-preserving morphism is isomorphic to ${\mathrm B}_X(S)$ for a suitable $X$.

Let ${\mathcal X}$ be a closed and invariant subset of $\widehat{E(S)}$ and ${I}_{\mathcal X}$ the ideal of the Boolean algebra $E({\mathrm B}(S))$ consisting of those compact-opens of $\widehat{{E(S)}}$ which do not intersect with ${\mathcal X}$.  Define the congruence $\sim_{\mathcal X}$ on ${\mathrm B}(S)$  by $a \sim_{\mathcal X} b$ if and only if there are $e,f,g\in E({\mathrm B}(S))$ such that 
${\bf d}(a) = e\vee f$, ${\bf d}(b) = e\vee g$ where $f,g\in I_{\mathcal X}$ and $ae=be$.
By \cite[Proposition 5.10]{LV19} it follows that ${\mathrm S}({\mathcal G}(S)|_{\mathcal X})$ is isomorphic to ${\mathrm B}(S)/\sim_{\mathcal X}$.

\begin{theorem}\label{th:quotients2} Let $S$ be an inverse semigroup.  Suppose that ${\mathcal X}$ is a closed invariant subset of $\widehat{E(S)}$. Then the quotient ${\mathrm B}(S) / \sim_{\mathcal X}$ is isomorphic to ${\mathrm B}_{X_{\pi_{\mathcal X}}}(S)$ where $\pi_{\mathcal X}\colon S\to {\mathrm B}(S)/\sim_{\mathcal X}$ is the composition of $\iota_{{\mathrm B}(S)}\colon S\to {\mathrm B}(S)$ and the quotient map ${\mathrm{B}}(S)\to {\mathrm B}(S)/\sim_{\mathcal X}$. \end{theorem}

\begin{proof}  
The restriction of $\pi_{\mathcal X}$ to $E(S)$ is the universal $X_{\pi_{\mathcal X}}$-to-join representation of $E(S)$, by Theorem \ref{th:isom}. Moreover, applying Lemma \ref{lem:character} it follows that ${\mathcal X}$ is precisely the space of \mbox{$X_{\pi_{\mathcal X}}$-to-join} characters of $E(S)$. Therefore $\pi_{\mathcal X}$ is the universal $X_{\pi_{\mathcal X}}$-to-join representation of $S$ and ${\mathrm B}_{X_{\pi_{\mathcal X}}}(S)\simeq {\mathrm B}(S)/\sim_{\mathcal X}$. 
\end{proof}

Applying the bijection between closed invariant subsets of $\widehat{E(S)}$ and weakly meet-preserving quotients of ${\mathrm B}(S)$ (described in \cite{LV19}), we obtain the following statement.

\begin{corollary}\label{cor:quotients1} Let $S$ be an inverse semigroup. Suppose that $\varphi\colon {\mathrm B}(S) \to T$ is a surjective weakly meet-preserving additive morphism where $T$ is a Boolean inverse semigroup. Then there is an invariant subset $X\subseteq E(S)\times {\mathcal P}_{fin}(E(S))$ such that $T \simeq {\mathrm B}_X(S)$.  
\end{corollary}

\section{Intermediate boundary quotients of $C^*$-algebras of Zappa-Sz\'ep product right LCM semigroups via $X$-to-join representations}\label{sect:intermediate}

\subsection{$X$-to-join representations of inverse semigroups in $C^*$-algebras}
Let $S$ be an inverse semigroup and $A$ a  $C^*$-algebra.

\begin{definition}\cite[Definition 10.4]{Exel:comb} A map $\sigma\colon S\to A$ is a {\em{representation}}  if the following conditions hold:
\begin{enumerate}
\item $\sigma(0)=0$;
\item $\sigma(st)=\sigma(s)\sigma(t)$ for all $s,t\in S$;
\item $\sigma(s^{-1})=(\sigma(s))^*$ for all $s\in S$.
\end{enumerate}
\end{definition}

Let $D_{\sigma}$ be the $C^*$-subalgebra of $A$ generated $\sigma(E(S))$ and let
$$
B_{\sigma} = \{e\in D_{\sigma}\colon e^2=e\}.
$$ 
It is a Boolean algebra with respect to the operations
$$
a\wedge b= ab, \,\, a\vee b=a+b-ab, \,\, a\setminus b=a-ab.
$$

The following is an adaptation of  \cite[Definition 13.1]{Exel:comb}.

\begin{definition} {\em Let $X\subseteq E(S)\times {\mathcal P}_{fin}(E(S))$.  A representation $\sigma$ of $S$ in $A$ is called $X$-to-join if the restriction of $\sigma$ to $E(S)$ is an $X$-to-join representation of $E(S)$ in the Boolean algebra $B_{\sigma}$.}
\end{definition}

\subsection{Right LCM semigroups and their left inverse hulls} A semigroup $P$ is called a {\em{right LCM semigroup}} if it is left cancellative and the intersection of any two principal right ideals is either empty or a principal right ideal. We assume that $P$ has the identity element denoted $1_P$. By ${\mathcal J}(P)$ we denote the set of all principal right ideals of $P$, plus the empty \mbox{set $\varnothing$.}

Let $P$ be a right $LCM$ semigroup and $U(P)$ the group of units of $P$, that is, the group of invertible elements of $P$. There is an equivalence relation, $\sim$, on $P\times P$, given by $(p,q)\sim (a,b)$ if and only if there is $u\in U(P)$ such that $p=au$ and $q=bu$. Let $[p,q]$ denote the $\sim$-class of $(p,q)$, for any $p,q\in P$. It is known (see, e.g., \cite[Proposition 3.2]{Starling15}) that the set 
$$
{\mathcal S}_P = \{[p,q]\colon p,q\in P\}\cup \{0\}
$$
is an inverse semigroup with the identity $[1_P, 1_P]$, the multiplication operation given by
$$
[a,b][c,d] = \left\lbrace\begin{array}{ll}[ab',dc'], & {\text{if }} bP\cap cP = rP {\text{ and }} bb'=cc'=r, \\
0, & {\text{if }} bP\cap cP = \varnothing, \end{array}\right.
$$
$s0=0s=0$ for all $s\in {\mathcal S}_P$, and the inversion operation given by $[p,q]^{-1}=[q,p]$. The idempotents of ${\mathcal{S}}_P$ are precisely the elements $[p,p]$ where $p$ runs through $P$.  It is easy to check  that the inclusion $pP\subseteq qP$ is equivalent to $[p,p]\leq [q,q]$, for all $p,q\in P$.
The inverse semigroup ${\mathcal S}_P$ is called the {\em left inverse hull} of $P$.

\subsection{The $C^*$-algebra and the boundary quotient $C^*$-algebra of a right LCM semigroup}

 The following is a special case of the definition due to Li~\cite{Li2012} (see also \cite{Norling}).

\begin{definition} Let $P$ be a countable right LCM semigroup. Then $C^*(P)$ is defined as the universal $C^*$-algebra generated by a set of isometries $\{v_p\colon p\in P\}$ and a set of projections $\{e_X\colon X\in {\mathcal{J}}(P)\}$ subject to the following relations:
\begin{enumerate}[(L1)]
\item $v_pv_q=v_{pq}$ for all $p,q\in P$;
\item $v_pe_{X}v_p^* = e_{pX}$ for all $p\in P$;
\item $e_{P} = 1$, $e_{\varnothing} =0$;
\item $e_Xe_Y = e_{X\cap Y}$, for all $X,Y\in {\mathcal J}(P)$.
\end{enumerate}
\end{definition}

It is easy to check (or see \cite[Lemma 3.4]{Starling15}) that the map $\sigma\colon {\mathcal S}_P\to C^*(P)$ given by 
\begin{equation}\label{eq:sigma}
\sigma([p,q]) = v_pv_q^*, \,\,  p,q\in P, \, \text{ and  } \sigma(0)=0
\end{equation} is a representation.

Norling \cite{Norling} showed that $C^*(P)$ is isomorphic to $C_u^*({\mathcal S}_P)$, the {\em universal $C^*$-algebra of the inverse semigroup} ${\mathcal S}_P$, that is, it is the universal $C^*$-algebra generarted by one element for each element of ${\mathcal S}_P$, such that the standard map $\pi_u\colon {\mathcal S}_P\to C^*({\mathcal S}_P)$ is a representation. It follows from \cite[4.4.1]{Paterson} that $C^*(P)$ is isomorphic to the groupoid $C^*$-algebra $C^*({\mathcal G}({\mathcal S}_P))$ of the universal groupoid ${\mathcal G}({\mathcal S}_P)$ of ${\mathcal S}_P$.

A finite subset $F\subset P$ is called a {\em foundation set} if for all $p\in P$ there exists $f\in F$ such that $fP\cap pP\neq \varnothing$. 

\begin{lemma}\label{lem:found} A subset $F\subset P$ is a foundation set if and only if the set $\{[f,f]\colon f\in F\}$ is a cover of $[1_P, 1_P]$ in $E({\mathcal S}_P)$.
\end{lemma}

\begin{proof} Suppose that $F$ is a foundation set and $a\in P$. Then there is $f\in F$ such that $aP\cap fP\neq \varnothing$.  This is equivalent to $[a,a][f,f]\neq 0$ which implies that $\{[f,f]\colon f\in F\}$ is a cover of $[1_P,1_P]$.
\end{proof}

The next definition is taken from \cite{BRRW}.

\begin{definition} Let $P$ be a countable right LCM semigroup. The {\em boundary quotient} ${\mathcal Q}(P)$ of $C^*(P)$ is the universal $C^*$-algebra generated by the sets $\{v_p\colon p\in P\}$ and $\{e_X\colon X\in {\mathcal J}(P)\}$ subject to relations (L1)--(L4) and 
$$
\prod_{p\in F}(1-e_{pP}) = 0 \,\, \text{for all foundation sets } F\subset P. 
$$
\end{definition}

There is another $C^*$-algebra associated to $P$, namely the {\em tight} $C^*$-algebra $C^*_{tight}({\mathcal S}_P)$. It is the universal $C^*$-algebra generated by one element for each element of ${\mathcal S}_P$ subject to the relations saying that the standard map $\pi_t\colon {\mathcal S}_P\to C_{tight}^*({\mathcal S}_P)$ is a tight representation, or, equivalently (see \cite{Exel19}), a cover-to-join representation. By a result of Exel \cite[Proposition~13.3]{Exel:comb} this $C^*$-algebra is  isomorphic to the $C^*$-algebra of the tight groupoid ${\mathcal G}_{tight}({\mathcal S}_P)$. It was proved by Starling \cite{Starling15} that the $C^*$-algebras $C^*_{tight}({\mathcal S}_P)$ and ${\mathcal Q}(P)$ are isomorphic.

\subsection{Zappa-Sz\'ep products: a construction due to Brownlowe, Ramagge, Robertson and Whittaker}
Let $P$ be a semigroup with unit $e$ and suppose that $U,A \subseteq P$ are subsemigroups such that the following conditions hold:
\begin{itemize}
\item $U\cap A = \{e\}$;
\item for all $p\in P$ there is a unique $(u,a)\in U\times A$ such that $p=ua$.
\end{itemize}
Then $P$ is called an {\em internal Zappa-Sz\'ep product} $U \bowtie A$ of $U$ and $A$.
If $u\in U$ and $a\in A$ we write $au = (a\cdot u)a|_u$ where $(a\cdot u) \in U$ and $a|_u\in A$. This data defines the {\em action} and the {\em restriction} maps.  Zappa-Sz\'ep products as above were introduced by Brin \cite{Brin}, where an equivalent, external, definition of  $U \bowtie A$ was found.
The paper \cite{BRRW} deals with Zappa-Sz\'ep products $P=U \bowtie A$ such that the following conditions hold:
\begin{itemize}
\item[(C1)] $U,A$ are right LCM;
\item[(C2)] ${\mathcal J}(A)$ is totally ordered by inclusion;
\item[(C3)] The map $u \mapsto a\cdot u$ is a bijection for each $a\in A$.
\end{itemize}
It is proved in \cite{BRRW} that in this case $P$ is right LCM.
Many interesting and extensively studied examples of semigroups appear to be decomposable as such Zappa-Sz\'ep products, see \cite{BRRW}, among which is the semigroup ${\mathbb N} \rtimes {\mathbb N^{\times}}$ (see~\cite{BaHLR} and references therein).

Suppose that axioms (C1), (C2) and (C3) hold. Brownlowe, Ramagge, Robertson, Whittaker \cite[Theorem 5.2]{BRRW} proved that
${\mathcal Q}(U \bowtie A)$ is a quotient of $C^*(U \bowtie A)$ by the following relations:
\begin{itemize}
\item[(Q1)] $e_{aP} = 1$ for all $a\in A$;
\item[(Q2)] $\prod\limits_{f\in F} (1- e_{fP}) = 0$ for all foundation sets $F\subseteq U$ of $U$.
\end{itemize}

Relations (Q1) and (Q2) are formulated in \cite[Theorem 5.2]{BRRW} slightly differently, but our formulation above is equivalent to theirs: if $P$ is a right LCM semigroup that is an internal Zappa-Sz\'ep product satisfying conditions (C1)-(C3) then $C^*(P)$ is generated by $v_u$, $u\in U$, and $v_a$, $a\in A$ and thus $t_u$ (respectively $s_a$) of \cite[Theorem 5.2]{BRRW} can be identified with $v_u$ (respectively $v_a$).

The following two quotients of $C^*(U\bowtie A)$ are also defined in \cite{BRRW}:
\begin{itemize}
\item $C^*_A(U\bowtie A)$ --- the quotient by relations (Q1) called the {\em additive boundary quotient};
\item $C^*_U(U\bowtie A)$ ---  the quotient by relations (Q2) called the {\em multiplicative boundary quotient}.
\end{itemize}
These quotients, for the semigroup ${\mathbb N}\rtimes {\mathbb N}^{\times}$, were studied earlier in \cite{BaHLR}. 

\subsection{$C^*_A(U\bowtie A)$ and $C^*_U(U\bowtie A)$ via $X$-to-join representations}

Let ${\mathcal S}_P$ denote the left inverse hull of $P=U\bowtie A$. We define the following subsets of $E({\mathcal S}_P) \times {\mathcal P}_{fin}(E({\mathcal S}_P))$:

\begin{itemize}
\item the set $X_A$ consists of all $([a,a], \{[b,b]\})$ where $a\in A$ and $b\in aA$;
\item the set $X_U$ consists of all $([s,s], \{[s_1,s_1], \dots, [s_n,s_n]\})$ where $s\in U$, $s_i\in sU$ for all $i\in \{1,\dots, n\}$ and for each  $t\in sU$  there is $i\in \{1,\dots, n\}$ satisfying $s_iU\cap tU\neq \varnothing$. 
\end{itemize}

We consider two more $C^*$-algebras associated to $P$:  $C^*_{A}({\mathcal S}_P)$, and $C^*_{U}({\mathcal S}_P)$. They are defined as the universal $C^*$-algebras generated by one element for each element of ${\mathcal S}_P$ subject to the following relations:
\begin{enumerate}
\item for $C^*_{A}({\mathcal S}_P)$ these are the relations saying that the standard map $\pi_A\colon {\mathcal S}_P\to C^*_A({\mathcal S}_P)$ is an $X_A$-to-join representation.
\item for $C^*_{U}({\mathcal S}_P)$ these are the relations saying that the standard map $\pi_U\colon {\mathcal S}_P\to C^*_U({\mathcal S}_P)$ is an $X_U$-to-join representation.
\end{enumerate}

\begin{theorem} \label{th:isomorphisms} Let $P$ be a countable right LCM semigroup which is decomposed as \mbox{$P=U\bowtie A$} and suppose that conditions (C1), (C2) and (C3) hold. The following pairs of $C^*$-algebras are isomorphic:
\begin{enumerate}
\item $C^*_A(U\bowtie A)$ and  $C^*_{A}({\mathcal S}_P)$;
\item $C^*_U(U\bowtie A)$ and   $C^*_{U}({\mathcal S}_P)$.
\end{enumerate}
\end{theorem}

\begin{proof}  Since $[a,a]\leq [1_P,1_P]$ for all $a\in A$, we have that $\pi_A([a,a])=\pi_A([1_P,1_P])=1$ for all $a\in A$. In view of (Q1) it follows that $C^*_{A}({\mathcal S}_P)$ is a canonical quotient of $C^*_A(U\bowtie A)$. Similarly, $C^*_{U}({\mathcal S}_P)$ is a canonical quotient of $C^*_U(U\bowtie A)$.
The map $\sigma\colon {\mathcal S}_P\to C^*(U\bowtie A)$ from \eqref{eq:sigma} induces the maps $\sigma_A\colon {\mathcal S}_P\to C_A^*(U\bowtie A)$ and $\sigma_U\colon {\mathcal S}_P\to C_U^*(U\bowtie A)$.

To complete the proof of part (1), we show that $C^*_A(U\bowtie A)$ is a canonical quotient of $C^*_{A}({\mathcal S}_P)$. Due to the universal property of $C^*_{A}({\mathcal S}_P)$ it suffices to show that $\sigma_A$ is an $X_A$-to-join representation. First note that $\sigma_A([1_P,1_P]) = e_P =1$ so by (Q1) we have $\sigma_A([a,a]) = e_{aP}=1$ for all $a\in A$. Let now $([a,a], \{[b,b]\})\in X_A$. Let $c\in A$ be such that $b=ac$. We then have
$$
\sigma_A([a,a]) = v_a1v_a^* = v_ae_{cP}v_a^*=e_{acP} = e_{bP} = v_bv_b^* = \sigma_A([b,b]),
$$
as needed.

To complete the proof of part (2), we show that $C^*_U(U\bowtie A)$ is a canonical quotient of $C^*_{U}({\mathcal S}_P)$. Due to the universal property of $C^*_{U}({\mathcal S}_P)$ it suffices to show that $\sigma_U$ is an $X_U$-to-join representation. In a similar way as in \cite[proof of Lemma 3.5]{Starling15} we see that for  any foundation set $F\subseteq U$ of $U$ (Q2) implies that $1= \bigvee_{f\in F} \sigma_U([f,f])$. Let now $([s,s], \{[s_1,s_1], \dots, [s_n,s_n]\})\in X_U$. For all $i=1,\dots, n$: since $s_i\in sU$ there is $u_i\in U$ such that $s_i=su_i$.
Let $u\in U$. Then $su \in sU$ and the definition of $X_U$ implies that there is $i\in \{1,\dots, n\}$ such that $s_iU\cap suU\neq \varnothing$, that is $su_iU \cap suU \neq \varnothing$. Thus $u_iU\cap uU\neq \varnothing$. We have shown that $\{u_1,\dots, u_n\}\subseteq U$ is a foundation set of $U$. It follows that $\bigvee_{i=1}^n \sigma_U([u_i,u_i]) = 1$.  In a similar way as in  \cite[proof of Lemma 3.5]{Starling15} we then have
\begin{equation*} \sigma_U([s,s]) = v_s1v_s^* = v_s\left(\bigvee_{i=1}^ne_{u_iP}\right)v_s^* = \bigvee_{i=1}^n v_se_{u_iP}v_i^* = \bigvee_{i=1}^n e_{su_iP} = \bigvee_{i=1}^n e_{s_iP} = \bigvee_{i=1}^n
\sigma_U([s_i,s_i]),
\end{equation*}
as needed. This completes the proof.
\end{proof}

Applying an adaptation of \cite[Theorem 13.2 and Theorem 13.3]{Exel:comb} to $X$-to join representations, we obtain the following statement.

\begin{corollary} \label{cor:iso} $C^*_A(U\bowtie A)$ and $C^*_U(U\bowtie A)$ are groupoid $C^*$-algebras.
\end{corollary}

\begin{remark} We remark that intermediate quotients for $C^*$-algebras of right LCM semigroups satisfying certain requirements were considered by Stammeier in \cite{Stammeier}. For a right LCM semigroup $S$ he introduced the core subsemigroup $S_c$ and the core irreducible subsemigroup $S_{ci}$ and considered the case where $S=S_{ci}^1S_c$. One can show, along  the lines of the proof of Theorem \ref{th:isomorphisms} that the intermediate boundary quotients considered in \cite{Stammeier} are isomorphic to universal $C^*$-algebras of ${\mathcal S}_S$ with respect to suitable $X$-to-join representations.  It follows that they are groupoid $C^*$-algebras, realizing a suggestion of \cite[p.~542]{Stammeier}.
\end{remark}

\end{document}